\documentclass[10pt]{amsart}
\usepackage{epsfig}
\usepackage{mathtools}
\usepackage[subnum]{cases}
\usepackage{enumerate}
\usepackage{bbm,bm}
\usepackage{amsmath,amssymb,mathrsfs,amsthm}
\usepackage{color}
\usepackage{hyperref}
\usepackage{verbatim}
\usepackage{tikz}
\usepackage{multicol}
\usepackage{multirow}
\usepackage{wrapfig}
\hypersetup{
    colorlinks,
    citecolor=black,
    filecolor=black,
    linkcolor=black,
    urlcolor=black
}

\usepackage{subfig}
\usepackage{listings}
\definecolor{gris245}{RGB}{245,245,245}
\definecolor{olive}{RGB}{50,140,50}
\definecolor{brun}{RGB}{175,100,80}

\usepackage{caption}

\newtheorem{lem}{Lemma}
\newtheorem{prop}{Proposition}
\newtheorem{cor}{Corollary}

\newtheorem*{remark*}{Remark}

\newtheorem{thm}{Theorem}

\newcommand{\sech}{\text{sech}}

\begin{document}

\title{A note concerning polyhyperbolic and related splines} 
	\author{Jeff Ledford*}
    \thanks{The funding for this work was provided by the Perspectives on Research In Science \& Mathematics (PRISM) program at Longwood University.}
        \author{Luke Paris}
	\author{Ryan Urban}
        \author{Alec Vidanes}
        
\address{Jeff Ledford\newline
	\indent Longwood University \newline
	\indent Department of Mathematics and Computer Science \newline
	\indent Farmville, VA, USA}
\email{ledfordjp@longwood.edu}

\address{Luke Paris\newline
	\indent Longwood University \newline
	\indent Department of Mathematics and Computer Science\newline
	\indent Farmville, VA, USA}
\email{luke.paris@live.longwood.edu}

\address{Ryan Urban\newline
	\indent Longwood University \newline
	\indent Department of Mathematics and Computer Science\newline
	\indent Farmville, VA, USA}
\email{ryan.urban@live.longwood.edu}

\address{Alec Vidanes\newline
	\indent Longwood University \newline
	\indent Department of Mathematics and Computer Science\newline
	\indent Farmville, VA, USA}
\email{alec.vidanes@live.longwood.edu}



\begin{abstract}
		
	 This note concerns the interpolation problem with two parametrized families of splines related to polynomial spline interpolation.  The authors address the questions of uniqueness and establish basic convergence rates for splines of the form $ s_\alpha = p\cosh(\alpha\cdot)+q\sinh(\alpha \cdot)$ and $t_\alpha = p+q\tanh(\alpha \cdot) $ between the nodes where $p,q\in\Pi_{k-1}$.
		\vspace{5pt}
		
		\vspace{5pt}
	\end{abstract}

\maketitle
\pagestyle{myheadings}
\markboth{Author Names}{Title of article}

\section{Introduction}

This note continues the study of splines, popularized by Schoenberg \cite{ Schoenberg_Part_A,SCHOENBERG,MR0053177}.  The goal of this paper is to develop basic properties of two parametrized families of splines.  Specifically, we address questions of uniqueness and convergence rate similar to those found in \cite{AHLBERG,Birkhoff_deBoor,Hall,Pruess,Spath}.  The first family of splines satisfy $s\in C^{2k-2}[a,b]$ and
\begin{equation}\label{polyhyperbolic}
(D^2-\alpha^2)^k s =0\quad\text{ on } [a,b]\setminus X,
\end{equation}
where $\alpha>0$ and $X$ is a partition of $[a,b]$.  Following \cite{ledford}, we call these \emph{k-th order polyhyperbolic} splines.  These splines are examples of $L$-splines, \cite[Ch. 10]{Schumaker_book}. The hyperbolic designation was chosen because the homogeneous solution of \eqref{polyhyperbolic} is given by 
\[s_\alpha(x)=p(x)\cosh(\alpha x)+q(x)\sinh(\alpha x),
\]
where $p$ and $q$ are polynomials of degree at most $k-1.$
We note that, in the literature, splines corresponding to the differential equation  $D^2(D^2-\alpha^2) u =0 $ and are often called hyperbolic but they are also referred to by the terms tension or exponential.  These splines have been studied extensively, see \cite{Maciejewski, McCartin, Pruess,Spath} among many others.  We also find a different notion for hyperbolic splines in \cite{Schumaker}, these correspond to the differential equation \[(D^2-(2k-1)^2)\cdots(D^2-3^2)(D^2-1)u=0 \quad\text{ or }\quad (D^2-(2k)^2)\cdots(D^2-2^2)Du=0.\]

Interpolation schemes which depend on parameters have appeared throughout the literature.  Polyhyperbolic splines are similar to (but distinct from) splines in tension \cite{Pruess} and the GB splines of \cite{Ksasov}, although the do share the property that in the appropriate limiting sense, these schemes lead to polynomial spline interpolation.  For the case studied here, we expect to recover cubic splines, since as $\alpha \to 0$, the differential operator reduces to $D^4$.

The second family of splines are built out of $\tanh$, specifically these are piecewise combinations of $t=p+q\tanh(\alpha \cdot)$, where $p$ and $q$ are polynomials.  We can characterize the order these by the maximal degree of the polynomials; the $k$-th order splines have polynomials whose degree does not exceed $k-1$.  Owing to the popularity of $\tanh$ as an activation function in neural networks \cite{aggarwal} and the use of splines for the same \cite{unser_deep,unser_representer}, it's possible these splines have application in their own right.  However in this note, we will concern ourselves with basic approximation properties of these splines.  In principal, interpolating with either spline of order $k$ corresponds to $2k-1$ degree \emph{polynomial} spline interpolation. We will explore the problem for the values $k=1$ and $k=2$, which correspond to classical linear and cubic interpolation, respectively.

The rest of the paper is laid out as follows.  Section 2 contains definitions and facts necessary to the sequel.  Section 3 contains results for first order splines, while Section 4 contains those related to the second order splines.  The main results are found in these sections as Proposition \ref{k=1 convergence} and Theorem \ref{k=2 convergence} and its corollary.  $B$-spline bases for low degree splines are developed in Section 5.  A related interpolation problem for higher degree splines is solved using a Fourier analysis approach in Section 6, whose main result is Theorem \ref{thm: high degree main}.


\section{Definitions and Basic Facts}
Throughout the sequel, we use standard notations for derivatives.  For example, $D^2u$ and $u''$ both correspond to the second derivative of the function $u$.  The set of polynomials of degree at most $k$ is denoted $\Pi_k$.  By a \emph{partition} of the interval $[a,b]$, we mean a sequence of increasing values $X=(x_j: 0\leq j\leq N)$ such that $a=x_0<\cdots<x_N=b$.  Associated to a fixed $X$, we have $h_j:=x_{j}-x_{j-1}$ and $\overline{h}(X):=\displaystyle\max_{1\leq j\leq N}h_j$.

We denote the set of all $k$-th order polyhyperbolic splines by $S_{\alpha}^k(X)$, that is
\[
S_{\alpha}^k(X):=\{ s\in C^{2k-2}[a,b]: (D^2-\alpha^2)^k s =0 \text{ on } [a,b]\setminus X \}.
\] 
We denote by $T_\alpha^k(X)$, the collection 
\[
T_{\alpha}^k(X) := \{ \sech(\alpha \cdot)u: u\in S_{\alpha}^k(X) \}.
\]
Note that on each interval $[x_{j},x_{j+1}]$, $t_\alpha \in T_{\alpha}^{k}(X)$ takes the form
\[
t_\alpha(x) = p(x)+ q(x)\tanh(\alpha x),
\]
where $p,q\in\Pi_{k-1}$.

We we use the notation $\| u \|_{L^\infty}$ to denote the usual $L^\infty$ norm for a function $u$, while we use the notation $\| M \|_{\infty}$ to mean the maximum row sum of the matrix $M$.

A \emph{diagonally dominant} matrix is an $N \times N$ square matrix $[a_{ij}]$ which satisfies $|a_{ii}|-\sum_{j\neq i}|a_{ij}|>0$ for each $1\leq i \leq N$.  The well known Levy-Desplanques theorem proves this condition implies the invertibility of $[a_{ij}]$, see \cite{Taussky}.

Generally we will seek to interpolate the data $Y=(y_j:0 \leq j\leq N)$ on a fixed, but otherwise arbitrary, partition $X$ using $s_\alpha\in S_{\alpha}^k(X)$ or $t_\alpha\in T_{\alpha}^k(X)$.  When we want to emphasize the dependence on the data set $Y$, we write $s_\alpha[Y]$ or $t_\alpha[Y]$.  Finally, we note the value of the constant $C$ will be depend on its occurrence and is typically independent of the listed parameters unless stated otherwise.  


\section{Properties of first order splines}
As one may expect, things are easier when $k=1$.  We include these results as more or less a warm up for the more involved $k=2$ case.  When $k=1$, we have $s_\alpha=\sum_{j=1}^{N}s_j$, where $s_j:[x_{j-1},x_{j}]\to\mathbb{R}$ is given by $s_j(x)=a_j\cosh(\alpha x)+b_j\sinh(\alpha x)$.  Thus we need only solve the $2\times 2$ matrix equation
\[
\begin{bmatrix} \cosh(\alpha x_{j-1}) & \sinh(\alpha x_{j-1})\\ 
\cosh(\alpha x_j) & \sinh(\alpha x_j)
\end{bmatrix}\begin{bmatrix}
    a_j\\b_j
\end{bmatrix} = \begin{bmatrix}
y_{j-1} \\y_{j}
\end{bmatrix}.
\]
The determinant of the matrix is
\[
\cosh(\alpha x_{j-1})\sinh(\alpha x_{j})   -  \cosh(\alpha x_{j})\sinh(\alpha x_{j-1}) = \sinh(\alpha h_j).
\]
Since $h_j>0$, this system has a unique solution which after simplification yields
\begin{equation}\label{linear solution}
s_j(x)= \dfrac{\sinh(\alpha(x_{j}-x))}{\sinh(\alpha h_j)}y_{j-1}+\dfrac{\sinh(\alpha(x- x_{j-1}))}{\sinh(\alpha h_j)}y_j.
\end{equation}
A similar computation yields
\begin{equation}\label{linear t solution}
t_{j}(x) = \dfrac{\tanh(\alpha x_{j})-\tanh(\alpha x)}{\tanh(\alpha x_{j})-\tanh(\alpha x_{j-1})}y_{j-1}+\dfrac{\tanh(\alpha x)-\tanh(\alpha x_{j-1})}{\tanh(\alpha x_{j})-\tanh(\alpha x_{j-1})}y_{j}.
\end{equation}
Expanding \eqref{linear solution} and \eqref{linear t solution} in a Taylor series (in $\alpha$) yields
\begin{align*}
s_\alpha\big|_{[x_{j-1},x_j]}(x)&=y_{j-1}+\dfrac{y_{j}-y_{j-1}}{h_j}(x-x_{j-1})+O(\alpha^2) \text{ and}\\
t_\alpha\big|_{[x_{j-1},x_j]}(x)&=y_{j-1}+\dfrac{y_{j}-y_{j-1}}{h_j}(x-x_{j-1})+O(\alpha^2) .
\end{align*}
We recognize the first two terms as the linear interpolant of $(x_{j-1},y_{j-1})$ and $(x_j,y_j)$.  Thus we expect $s_\alpha$ to exhibit similar convergence properties to the linear spline $l[Y]$.  The expansions above give us $s_0[Y]=t_0[Y]=l[Y]$.  
We may also estimate the rate of convergence by expanding a sufficiently smooth function $f$ in a Taylor series about $x=x_{j-1}$, then on $(x_{j-1},x_j)$ we have
\begin{align*}
f(x)-t_\alpha(x)&=\\ & \left(f'(x_{j-1}) - \alpha h_j(\coth(\alpha h_j)+\tanh(\alpha x_j))f[x_{j-1},x_j] \right)(x-x_{j-1})\\
&\quad+O(\overline{h}^2)\\
&= (f'(x_{j-1})-f[x_{j-1},x_j])(x-x_{j-1})+O(\overline{h}^2)\\
&=O(\overline{h}^2).
\end{align*}
The last equation requires $f\in C^2[a,b]$, and follows from the Mean Value Theorem.  A completely analogous result holds for $s_\alpha$.  We summarize these findings in the following propositions.

\begin{prop}\label{k=1 uniqueness}
Let $X$ be a partition of $[a,b]$, for any $\alpha>0$ and data set $Y$, the interpolants $s_\alpha[Y]$ and $t_\alpha[Y]$ are unique.  Furthermore they satisfy
\[
\lim_{\alpha\to 0}s_\alpha[Y](x)=l[Y](x) \quad \text{ and } \quad
\lim_{\alpha\to 0}t_\alpha[Y](x)=l[Y](x)
\]
for all $x\in [a,b]$, where $l[Y]$ is the linear spline interpolant.
\end{prop}

\begin{prop}\label{k=1 convergence}
Suppose that $f\in C^2[a,b]$ and $X$ is a partition of $[a,b]$.  If $t_\alpha\in T_{\alpha}^1(X)$ such that $t_\alpha\mid_{X} =f\mid_{X}$, then for $\alpha$ sufficiently small, we have
\[
\| f- t_\alpha  \|_{L^{\infty}[a,b]} \leq C \overline{h}^2,
\]
where $C>0$ depends on $f$ and $\alpha$.  The same is true if $t_\alpha$ is replaced by $s_\alpha\in S_{\alpha}^1(X)$.
\end{prop}
\begin{remark*}
Rather than expanding $s_\alpha$ in a Taylor series, we could use the fact that for any $f\in C^2[a,b]$, there exists $g\in C^2[a,b]$ with $f(x)=\cosh(\alpha x)g(x)$.  Thus
\[
\| f - s_\alpha  \|_{L^\infty [a,b]} = \| \cosh(\alpha \cdot) \left( g- t_\alpha \right)  \|_{L^{\infty}[a,b]} \leq C\overline{h}^2
\]
for sufficiently small $\alpha$.
\end{remark*}


\section{Properties of second order splines}
The $k=2$ problem is more involved.  Just as in the case with cubic splines, we must impose endpoint conditions on our splines $s_{\alpha} \in S_{\alpha}^2(X)$ (or $t_{\alpha} \in T_{\alpha}^2(X)$).  For these, we use those found in \cite{Pruess}:
\begin{align*}
\text{Type I: }&  s_\alpha '(a) = y'_0 , s_\alpha '(b)=y'_N ,\\
\text{Type II: }&  s_\alpha ''(a)=s_\alpha ''(b)=0, \\
\text{Type III: }&   s_\alpha ''(a)=y''_0 , s_\alpha ''(b)=y''_N.
\end{align*}
Assuming that we are sampling a function $f:[a,b]\to\mathbb{R}$, so that $y_j=f(x_j)$, conditions $I$ and $III$ become
\begin{align*}
\text{Type I: }&  s_\alpha '(a) = f'(a) , s_\alpha '(b)=f'(b) ,\\
\text{Type III: }&   s_\alpha ''(a)=f''(a) , s_\alpha ''(b)=f''(b).
\end{align*}

In light of the remark that follows Proposition \ref{k=1 convergence}, we find it more convenient to work with $t_\alpha \in T_{\alpha}^2(X)$ in what follows, often suppressing the dependence on $\alpha$.  Note that since $t_\alpha\mid_{[x_{j-1},x_j]}(x)=p_j(x)+q_j(x)\tanh(\alpha x)$, where $p_j,q_j\in \Pi_1$, $t''_\alpha\mid_{[x_{j-1},j]} = D^2[q_j\tanh(\alpha \cdot)]$.  Specifying $t_\alpha''(x_{j-1})$ and $t_{\alpha}''(x_j))$ allow us to solve for the coefficients of $q_j$ then  integrating twice and using the interpolation conditions allow us to solve for the coefficients of $p_j$. One may then generate a tridiagonal system to enforce smoothness of the first derivative, setting $t''_j:=t''_\alpha(x_j)$ we have
\begin{subequations}\label{tridiag}
\begin{align}
\label{tridiag 1}    b_0t''_0+c_0 t''_{1} &= d_0,\\
 \label{tridiag 2}   a_jt_{j-1}''+ b_jt''_j+c_jt''_{j+1} &= d_j; \quad 1\leq j\leq N-1,\\
 \label{tridiag 3}    a_{N}t''_{N-1}+b_{N}t''_{N} &= d_N.
\end{align}
\end{subequations}
One needs to evaluate $t'\mid_{[x_{j-1},x_j]}$ at both endpoints to generate the coefficients in \eqref{tridiag 2}.  Setting
\[
E_{j-1} = \dfrac{\alpha h_{j} \cosh(\alpha h_{j})-\sinh(\alpha h_{j})  }{2\alpha^2 h_{j} (\tanh(\alpha x_j)-\tanh(\alpha x_{j-1})-\alpha h_j\tanh(\alpha x_{j-1})\tanh(\alpha x_j))},
\]
we have for $1\leq j \leq N-1$
\begin{align*}
    a_{j} &=\dfrac{\cosh(\alpha x_{j-1})}{\cosh(\alpha x_j)}E_{j-1}=\dfrac{h_j}{6}+O\left((\alpha \overline{h})^2\right); \alpha\to 0,\\
    c_{j} &=\dfrac{\cosh(\alpha x_{j+1})}{\cosh(\alpha x_j)}E_{j}=\dfrac{h_{j+1}}{6}+O\left((\alpha \overline{h})^2\right); \alpha\to 0, \text{ and}\\
    d_{j} &= \dfrac{y_{j+1}-y_j}{h_{j+1}} - \dfrac{y_{j}-y_{j-1}}{h_j}.
\end{align*}
The term $b_j$ is more complicated, setting
\[
F_{j-1} = 4\alpha^2 h_{j}(\tanh(\alpha x_j)-\tanh(\alpha x_{j-1})-\alpha h_j\tanh(\alpha x_{j-1})\tanh(\alpha x_j)),
\]
which is twice the denominator of $E_{j-1}$, we can simplify the expression somewhat
\begin{align*}
b_{j} =&\left( \sinh(2\alpha x_j)-2\alpha h_j - 2\tanh(\alpha x_{j-1})(\cosh^2(\alpha x_j)-\alpha^2 h_{j}^2)  \right) / F_{j-1}\\
&-\left( \sinh(2\alpha x_j)+2\alpha h_{j+1} - 2\tanh(\alpha x_{j+1})(\cosh^2(\alpha x_j)-\alpha^2 h_{j+1}^2)  \right) / F_{j}\\
=& \dfrac{h_{j}+h_{j+1}}{3} +O\left((\alpha \overline{h})^2\right); \alpha\to 0.
\end{align*}

For type $I$ endpoint conditions, the coefficients in \eqref{tridiag 1} and \eqref{tridiag 3} are given by
\begin{align*}
b_0 &= -\left( \sinh(2\alpha x_0) +2\alpha h_1 - 2\tanh(\alpha x_{1})(\cosh^2(\alpha x_0)-\alpha^2 h_{1}^2)  \right)/F_{0}  \\
&=\dfrac{h_1}{3}+O\left((\alpha \overline{h})^2\right); \alpha\to 0,\\
c_0 &= \dfrac{\cosh(\alpha x_1)}{\cosh(\alpha x_0)}E_0 = \dfrac{h_1}{6}+O\left((\alpha \overline{h})^2\right);\alpha\to 0, \\
d_0&= \dfrac{y_1-y_0}{h_1}-y'_0,\\
a_N &=  \dfrac{\cosh(\alpha x_{N-1})}{\cosh(\alpha x_N)}E_{N-1}= \dfrac{h_N}{6}+O\left((\alpha \overline{h})^2\right);\alpha\to 0, \\
b_{N} &= \left( \sinh(2\alpha x_N)-2\alpha h_N - 2\tanh(\alpha x_{N-1})(\cosh^2(\alpha x_N)-\alpha^2 h_{N}^2)  \right)  /F_{N-1} \\
&= \dfrac{h_N}{3}+O\left((\alpha \overline{h})^2\right);\alpha\to 0, \text{ and}\\
d_{N} &=  y'_N - \dfrac{y_N-y_{N-1}}{h_N}.
\end{align*}

The adjustments for the other types are 
\begin{align*}
II: & b_0=b_N=1 , a_0=c_0=0, d_0=d_N=0\\
III: & b_0=b_N=1 , a_0=c_0=0, d_0=y''_0 , d_N=y''_N.
\end{align*}

Direct computation now provides the following analog of Proposition \ref{k=1 uniqueness}.

\begin{thm}\label{k=2 uniqueness}
Let $X$ be a partition of $[a,b]$.  For sufficiently small $\alpha$ and data set $Y$ with end condition type $I,II,$ or $III$, the interpolants $s_\alpha[Y] \in S_{\alpha}^2(X)$ and $t_\alpha[Y]\in T_{\alpha}^2(X)$ are unique.  Furthermore,
\[
\lim_{\alpha\to 0}s_\alpha[Y](x)=\sigma[Y](x) \quad\text{ and }\quad
\lim_{\alpha\to 0}t_\alpha[Y](x)=\sigma[Y](x)
\]
for all $x\in [a,b]$, where $\sigma[Y]$ is the cubic spline interpolant with the same type of endpoint condition.
\end{thm}

\begin{proof}
The matrix \eqref{tridiag} for $t_\alpha \in T_{\alpha}^{2}(X)$ is diagonally dominant for sufficiently small $\alpha$. Now consider $\widetilde{Y}=(\sech(\alpha x_j)y_j:0\leq j\leq N)$, we have a unique $\tilde{t}_\alpha\in T_{\alpha}^2(X)$ which interpolates $\widetilde{Y}$ and $s_\alpha = \cosh(\alpha \cdot)t_\alpha \in S_{\alpha}^2 (X)$ interpolates $Y$.
\end{proof}

We use an argument similar to the one given in \cite{Pruess} to prove a result similar to Proposition \ref{k=1 convergence}.  We begin by setting $\delta=\sigma - t_\alpha$ and $\delta_{j}'' = \delta''(x_j)$.  We use \eqref{tridiag} and the tridiagonal system for $\sigma$ to generate a system for $\delta''$:
\begin{subequations}\label{delta tridiag}
\begin{align}
\label{delta tridiag 1}    &\delta''_0+\frac{c_0}{b_0} \delta''_{1} = \frac{3b_0-h_1}{3b_0}\sigma''(x_0)+\frac{6c_0-h_1}{6b_0}\sigma''(x_1),\\
 \label{delta tridiag 2} \begin{split}   &\frac{a_j}{b_j}\delta''_{j-1}+\delta''_j+\frac{c_j}{b_j}\delta''_{j+1} \\
 &=\frac{6a_j-h_j}{6b_j}\sigma''(x_{j-1})+\frac{3b_j-h_j-h_{j+1}}{3b_j}\sigma''(x_{j})+ \frac{6c_j-h_{j+1}}{6b_j}\sigma''(x_{j+1}),\end{split}\\
 \label{delta tridiag 3} &   \frac{a_{N}}{b_N}\delta_{N-1}''+\delta_{N}'' = \frac{6a_N-h_N}{6b_N}\sigma''(x_{N-1})+\frac{3b_N-h_N}{3b_N}\sigma''(x_{N}).
\end{align}
\end{subequations}

Now we may write this as the matrix equation $(I+M)\delta''= b$, where $\delta''$ represents the vector with components $\delta_j''$ and $b$ is the vector version of the right hand side \eqref{delta tridiag}.  More computation with the coefficients in \eqref{tridiag} shows that $\| M\|_{\infty}=O(\frac{1}{2})$ as $\alpha\to 0$, hence for sufficiently small $\alpha$
\[
\| \delta'' \|_{\infty}=\| (I+M)^{-1}b  \|_{\infty} \leq 3 \| b \|_{\infty}.
\]
To estimate $\| b\|_{\infty}$ we note the Mean Value theorem provides the estimate for the first and last rows $O\left((\alpha \overline{h})^2\right)\| \sigma'''  \|_{L^{\infty}([a,b]\setminus X)}$, while the the triangle inequality provides the estimate for the middle rows $O\left((\alpha \overline{h})^2\right)\| \sigma'' \|_{L^{\infty}(X)}$ as $\alpha\to 0$.  Thus we have for sufficiently small $\alpha$
\begin{equation}\label{delta inequality}
\| \delta'' \|_{\infty}\leq  C(\alpha\overline{h})^2 \max\{\| \sigma'''  \|_{L^{\infty}([a,b]\setminus X)},\| \sigma''  \|_{L^{\infty}(X)}\}.
\end{equation}

Now the argument follows a similarly to Proposition \ref{k=1 convergence}.  We may use Rolle's theorem for the function $\delta$, which allows us to find $\xi_j\in(x_{j-1},x_{j})$ such that $\delta'(\xi_j)=0$, this together with the fact that $\delta(x_j)=0$  allow us to write
\[
\delta\big|_{[x_{j-1},x_{j}]}(x)=\int_{x_j}^x \delta'(u){\rm d}u \quad \text{ and } \quad \delta'\big|_{[x_{j-1},x_{j}]}(x)=\int_{\xi_j}^x \delta''(u){\rm d}u. 
\]
We need only estimate $\|\delta''  \|_{L^\infty[a,b]}$.  To this end, note that we can write things in terms the values at the partition
\begin{align*}
\delta''(x) &= \sigma''(x)-t_{\alpha}''(x)\\
&=w_{1}(x)\delta''_{j-1} +w_{2}(x)\delta''_{j}+z_{1}(x)\sigma''_{j-1} +z_{2}(x)\sigma''_{j},
\end{align*}
where
\begin{align*}
w_{1}(x) &= \dfrac{4\alpha^2 h_{j}^2\cosh(\alpha x_{j-1})((-1+\alpha(x_{j}-x)\tanh(\alpha x_{j}))\tanh(\alpha x)+\tanh(\alpha x_{j}))}{\cosh^2(\alpha x)F_{j-1}}, \\
z_{1}(x) &= \dfrac{x_j-x}{h_j} -w_{1}(x)= O\left( (\alpha \overline{h} )^2 \right) \| \sigma'''  \|_{L^{\infty}([a,b]\setminus X)};\alpha\to 0\\
w_{2}(x)&=\dfrac{4\alpha^2 h_{j}^2\cosh(\alpha x_j)((1+\alpha(x-x_{j-1})\tanh(\alpha x_{j-1}))\tanh(\alpha x)-\tanh(\alpha x_{j-1}))}{\cosh^2(\alpha x)F_{j-1}}\\
z_{2}(x) &= \dfrac{x-x_{j-1}}{h_j} -w_{2}(x).
\end{align*}
Thus
\[
\| \delta'' \|_{L^{\infty}[a,b]}= O\left( (\alpha \overline{h} )^2 \right) \max\left\{ \| \sigma'''  \|_{L^{\infty}([a,b]\setminus X)},\| \sigma''  \|_{L^{\infty}(X)}    \right\};\alpha\to 0.
\]
Putting these estimates together yields the following theorem.

\begin{thm}\label{k=2 convergence}
    Let $X$ be a partition for $[a,b]$.  Given a data set $Y$, suppose that $t_\alpha[Y]\in T_{\alpha}^{2}(X)$ interpolates $Y$ with type I (II, III) end conditions.  For $i=0,1,$ or $2$ and sufficiently small $\alpha$,
    \[
    \| D^i(\sigma[Y]-t_\alpha[Y])  \|_{L^{\infty}[a,b]}\leq C (\alpha \overline{h})^{4-i}\max\left\{ \| \sigma'''  \|_{L^{\infty}([a,b]\setminus X)},\| \sigma''  \|_{L^{\infty}(X)}    \right\},
    \]
    where $\sigma[Y]$ is the cubic spline interpolant with type I (II, III) end conditions and $C>0$ is a constant independent of $\overline{h}$.
\end{thm}

\begin{cor}\label{k=2 corollary}
 Let $X$ be a partition for $[a,b]$ and $f\in C^4[a,b]$.  Suppose that $t_{\alpha}[f]:=t_\alpha[f(X)]\in T_{\alpha}^{2}(X)$ interpolates $f(X)$ with type I (II, III) end conditions.  For $i=0,1,$ or $2$ and sufficiently small $\alpha$,
    \[
    \| D^i(t_\alpha[f]-f)  \|_{L^{\infty}[a,b]}\leq C \overline{h}^{4-i},
    \]
    where $C:=C_{\alpha ,f, X}>0$.  The same is true for $s_\alpha[f]\in S_{\alpha}^2(X)$.
\end{cor}

\begin{proof}
The result for $t_\alpha[f]$ follows from Theorem \ref{k=2 convergence}, the triangle inequality, and known convergence rates for $\sigma[f]:=\sigma[f(X)]$:
\begin{align*}
  \| D^i(\sigma[f]-f)  \|_{L^{\infty}[a,b]} &\leq C_{f,X} \overline{h}^{4-i};\quad i=0,1,2 \quad {\text{and}}\\
 \| D^3(\sigma[f]-f)  \|_{L^{\infty}[a,b]} &\leq C_{f,X} \overline{h} 
 \end{align*}
found in \cite{Hall} and \cite{Birkhoff_deBoor}, respectively.  The constants depend on the norms of various derivatives of $f$ and the \emph{mesh ratio}, $|X|=\overline{h}/\min{h_j}$.
To see the result for $s_\alpha[f]\in S_{\alpha}^2(X)$, note that any function $f\in C^4[a,b]$ may be written $f(x)=\cosh(\alpha x)g(x)$, where $g\in C^4[a,b]$ as well.  We have
\[
\| D^{i}(s_{\alpha}[f] - f) \|_{L^{\infty}[a,b]} =  \| D^{i}(\cosh(\alpha \cdot)(t_{\alpha}[g] - g)) \|_{L^{\infty}[a,b]} 
\]
so the result follows from the Leibniz rule and the corresponding result for $t_\alpha\in T_{\alpha}^2(X)$.
\end{proof}


\begin{remark*}
One powerful aspect of modeling with cubic splines is the ability to preserve certain qualities of the underlying data set to be interpolated.  Unfortunately, $s_\alpha[Y]\in S_{\alpha}^2(X)$ (or $t_\alpha[Y]\in T_{\alpha}^2(X)$) need not share properties of $Y$ such as positivity, monotonicity, or convexity.  As seen with interpolation with cubic splines, the trade off for preserving the shape of the data is giving up some smoothness of the interpolant.  If we choose to specify the first derivatives rather than enforcing continuity of the second derivative at each internal $x_j\in X$, then we can solve \eqref{polyhyperbolic} in terms of the values $Y=(y_j)$ and $Y'=(y'_j)$ on the interval $[x_{j-1},x_{j}]$:
\begin{equation}\label{Hermite}
\begin{split}
s_{\alpha}(x) &= y_j+y'_j(x-x_j)\\
&\quad+\left( \dfrac{3(y_{j+1}-y_j)-h_{j+1}(2y'_j+y'_{j+1})}{h_{j+1}^2}+O(\alpha^2)\right)(x-x_j)^2\\
&\qquad+\left(\dfrac{h_{j+1}(y'_j+y'_{j+1})-2(y_{j+1}-y_j)}{h_{j+1}^3}+O(\alpha^2)\right)(x-x_j)^3+O(\alpha^2)\\
&=\sigma(x)+O(\alpha^2 h_{\max}^2).
\end{split}
\end{equation}
Here $\sigma$ is the cubic Hermite spline interpolant for the same data.  In order for $s_\alpha$ to be approximately shape preserving, we may first generate a cubic interpolant with the property.  Algorithms to preserve positivity, monotonicity, and convexity for cubic splines have been well established.  See, for instance, \cite{Dougherty} and the references therein.
Essentially, these algorithms provide an interval from which to choose each parameter $y'_j\in Y'$ in such a way that certain inequalities hold for $\sigma$.  In order to ensure that $s_\alpha$ has the same property, we first choose the parameters of $\sigma$ to satisfy the strict inequalities.  This gives us room to reduce $\alpha$ small enough so that the error term does not affect the relevant inequality for $\sigma$.
\end{remark*}
\section{Local Bases}
In this section we develop a $B$-spline basis for both families $S_\alpha ^k(X)$ and $T_\alpha ^k(X)$, for $k=1,2$. These bases will have minimal support, in the $k=1$ case, we need 2 intervals of the form $[x_j,x_{j+1}]$ while the $k=2$ case requires 4 such intervals.  To see this we need the following result.
\begin{lem}\label{support}
Let $f\in S_\alpha ^k(X)$ (or $T_\alpha ^k(X)$), for $k=1,2$, with appropriate endpoint conditions specified.   If $\text{supp}(f)\subset [x_{j_0},x_{j_0+m}]$ where $m<2k$, then $f=0$.
\end{lem}
\begin{proof}
If $m<2k$, we get a system of linear equations whose only solution is $f=0$.  When $k=1$, we have the equations $f(x_0)=f(x_1)=0$, solving the invertible system of equations for the coefficients yields $f=0$.  The $k=2$ case works similarly.  Suppose that $m=3$.  Then we would need to solve for 12 coefficients using the 6 endpoint conditions: $f^{(i)}(a)=f^{(i)}(b)=0$, for $i=0,1,2$ and the 6 interior smoothness conditions.  Hence the result follows from Theorem \ref{k=2 uniqueness} and the fact that the system of equations for the coefficients is linear and homogeneous.
\end{proof}

Let us now build the bases for $k=1$.  Interior bases ($1\leq j \leq N-1$) are of the form
\[
t_{\alpha, j} (x)= \begin{cases} \dfrac{ \tanh(\alpha x) -\tanh(\alpha x_{j-1})}{\tanh(\alpha x_{j})-\tanh (\alpha x_{j-1})}, &  x \in [x_{j-1},x_j] \\
 \dfrac{ \tanh(\alpha x_{j+1}) -\tanh(\alpha x)}{\tanh(\alpha x_{j+1})-\tanh (\alpha x_{j})}, & x\in [x_j,x_{j+1}]
\end{cases}
\]
and
\[
s_{\alpha, j} (x)= \begin{cases} \dfrac{ \sinh(\alpha (x-x_{j-1})) }{\sinh(\alpha h_j)}, &  x \in [x_{j-1},x_j] \\
 \dfrac{ \sinh(\alpha (x_{j+1}-x)) }{\sinh(\alpha h_{j+1})}, & x\in [x_j,x_{j+1}].
\end{cases}
\]
and satisfy $t_{\alpha, j} (x_k) = s_{\alpha, j} (x_k)=\delta_{j,k}$, $0\leq k \leq N$.  The graphs are shown below.

\begin{figure}[h]
    \centering
        \includegraphics[scale=0.33]{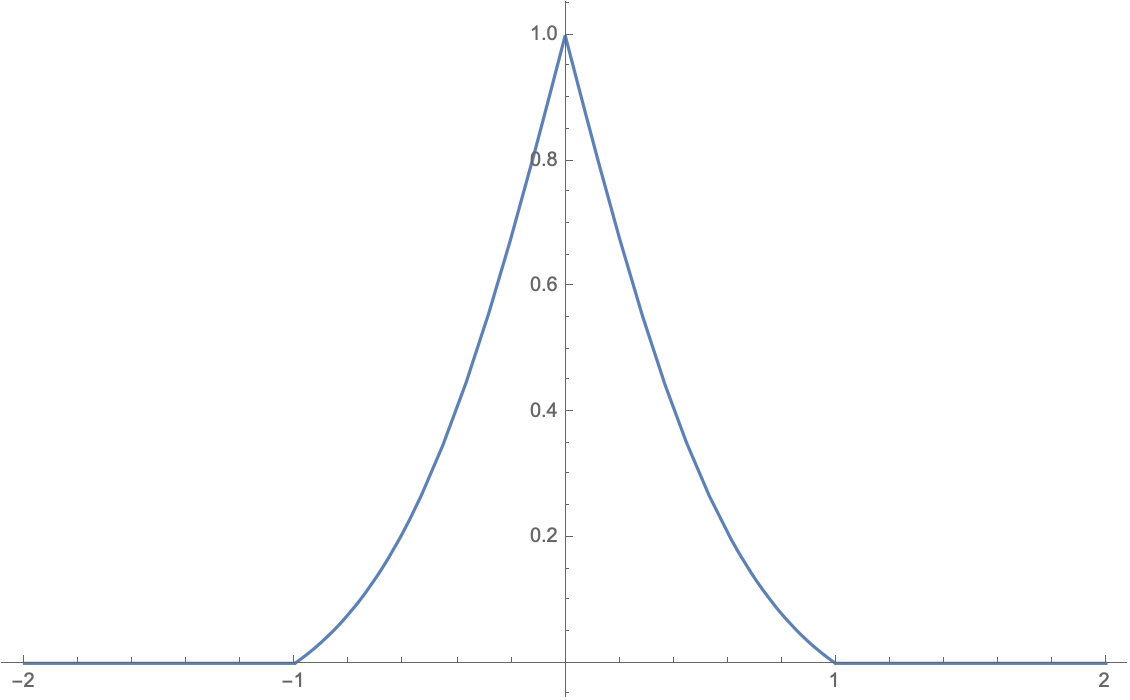}\includegraphics[scale=0.33]{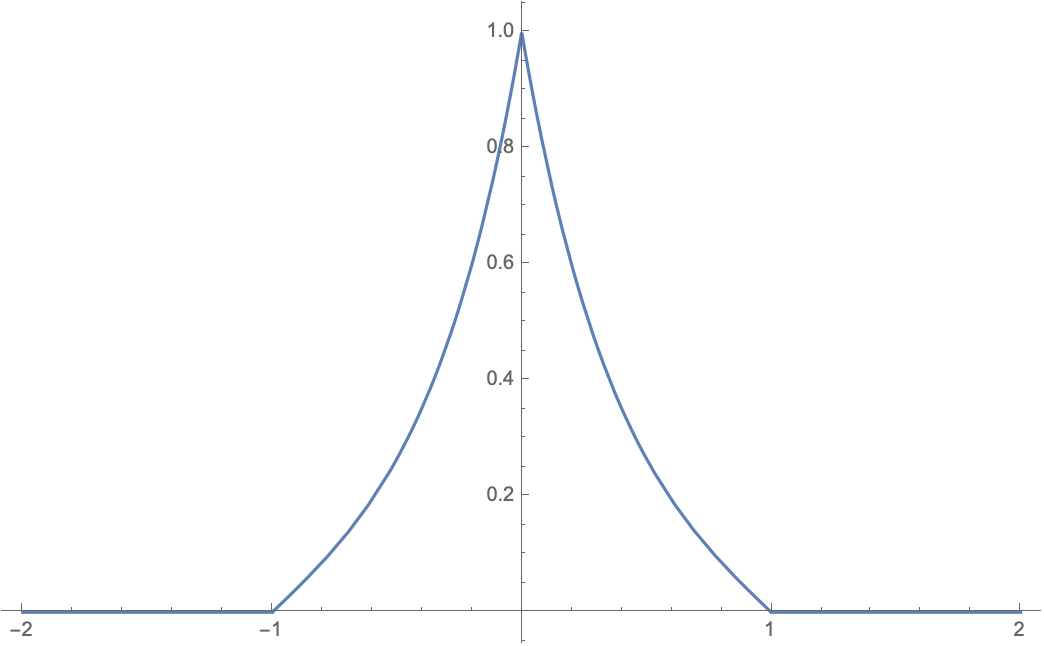}
        \caption{$t_{1.5,0}(x)$ and $s_{2.5,0}(x)$ plotted with $h_j=1$}
\end{figure}

For $k=2$, we present the uniform case $h_j\equiv h$ centered at $x=0$.   We have 
\[
s_{\alpha}(x) = \sum_{l=-2}^{1}\dfrac{s_{l}(x)}{e^{4\alpha h}-4\alpha h e^{2\alpha h}-1}\chi_{[lh,(l+1)h]}(x),
\]
which is an even function, with 
\begin{align*}
    s_{-2}(x) &= e^{-\alpha x}(1+\alpha(2h+x))+e^{4\alpha h +\alpha x}(-1+\alpha(2h+x)),\\
    s_{-1}(x) &= -e^{-\alpha x}(1+\alpha x)-2e^{2\alpha h -\alpha x}(1+\alpha(h+x))\\
    &\qquad -2e^{2\alpha h +\alpha x}(-1+\alpha(h+x)) +e^{4\alpha h -\alpha x}(1-\alpha x),\\
    s_{0}(x)  &= -e^{\alpha x}(1-\alpha x)-2e^{2\alpha h +\alpha x}(1+\alpha(h-x))\\
    &\qquad -2e^{2\alpha h -\alpha x}(-1+\alpha(h-x)) +e^{4\alpha h -\alpha x}(1+\alpha x),\\
    s_{1}(x) &= e^{\alpha x}(1+\alpha(2h-x))+e^{4\alpha h -\alpha x}(-1+\alpha(2h-x)).
\end{align*}

Similarly for $t_\alpha$ we have 
\[
t_\alpha(x) = \sum_{l=-2}^{1}\dfrac{t_l(x)}{2\alpha h -\sinh(2\alpha h)}\chi_{[lh,(l+1)h]}(x)
\]
where
\begin{align*}
t_{-2}(x) &= \sech(\alpha x)\left[\sinh(\alpha(2h+x))-\alpha(2h+x)\cosh(\alpha(2h+x)) \right]\\
    t_{-1}(x) &= 2\alpha h+2\alpha x+\alpha x\sech(\alpha x)\cosh(\alpha(2h-x))\\&\qquad-\sech(\alpha x)\sinh(\alpha(2h+x))- 2\tanh(\alpha x),\\
    t_{0}(x)  &= -\sech(\alpha x)\left[\alpha x \cosh(\alpha(2h-x))-2\alpha(h-x)\cosh(\alpha x)\right.\\&\qquad\left.+\sinh(\alpha(2h-x))-2\sinh(\alpha x)  \right],\\
    t_{1}(x) &= \sech(\alpha x)\left[\sinh(\alpha(2h-x))-\alpha(x-2h)\cosh(\alpha(2h-x)) \right].
\end{align*}

\begin{figure}[h]
    \centering
        \includegraphics[scale=.4]{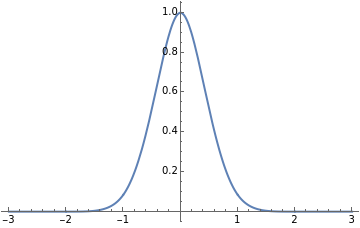}\includegraphics[scale=.4]{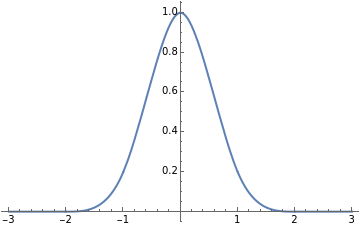}
        \caption{$t_{1.5,0}(x)$ and $s_{1.5,0}(x)$ plotted with $h_j=1$}
\end{figure}

In general, for a $k$-th order spline, we could solve the $(2k^2)\times (2k^2)$ system
\begin{align*}
s^{(l)}(x_{j-k})  = 0; \quad &0\leq l \leq 2k-2 \\
s^{(l)}(x_{m}-) = s^{(l)}(x_{m}+); \quad &0\leq l \leq 2k-2 , \quad j-k+1\leq m \leq j-1 \\
s^{(l)}(x_j) =\delta_{0,l};\quad &0\leq l \leq k-1.
\end{align*}
 then extend to the right side of $x_j$.

Note that $x\in[a,b]$, is in the support of at most $2k$ $B$-splines.  For $x_j\in X$, we need only $2k-1$ $B$-splines.


\section{Higher order splines}

Having exhibited properties for splines of order $k=1$ and $k=2$, one is naturally led to seek results for higher order splines $S_\alpha^k(X)$ and $T_\alpha^k(X)$, where $k>2.$  The expectation is that properties from polynomial splines carry over for sufficiently small $\alpha>0$.  In the remainder of this note, we prefer to explore the interpolation problem when $X$ is allowed to be infinite.  This problem was studied in \cite{ledford}, and our goal is to establish the existence of spline interpolants for scattered data.  

In particular, we assume that $X=(x_j:j\in\mathbb{Z})$ is a \emph{complete interpolating sequence} (CIS) which means that the sequence $(e^{ix_j\cdot}:j\in\mathbb{Z})$ is a Riesz basis for $L
_2([-\pi,\pi])$.  Among other things, we have the \emph{Riesz basis inequality}
\begin{equation}\label{Riesz basis inequality}
    C_X^{-1}\sum_{j\in\mathbb{Z}}|a_j|^2 \leq \int_{-\pi}^{\pi} \left|\sum_{j\in\mathbb{Z}}a_j e^{-ix_j\xi}\right|^2 {\rm d}\xi \leq C_X \sum_{j\in\mathbb{Z}}|a_j|^2, 
\end{equation}
for some $C_X>0$ and all $(a_j)\in \ell_2$.  Associated to a Riesz basis $X$, we define for $n\in\mathbb{Z}$, the \emph{prolongation operator} $A_n: L_2([-\pi,\pi]) \rightarrow L_2([-\pi,\pi])$ by
\begin{equation}\label{prolongation operator}
    A_n[f](\xi)=A_n\left(\sum_{j\in\mathbb{Z}}a_je^{-ix_j\xi}\right):=\sum_{j\in\mathbb{Z}}a_je^{-2\pi i n x_j}e^{-ix_j\xi}; \quad |\xi| \leq \pi.
\end{equation}
In light of \ref{Riesz basis inequality}, the operators $A_n$ and their adjoints $A_n^*$ are uniformly bounded.

We will make extensive use of the \emph{Fourier transform}, which for $g\in L_1(\mathbb{R})$ is given by
\[
\hat{g}(\xi):=\int_{\mathbb{R}}g(x)e^{-ix\xi}{\rm d}x,
\]
and its extension to $g\in L_2(\mathbb{R})$ is denoted $\mathcal{F}[g]$.  Using this convention, the inversion formula is given by
\[
g(x)=\frac{1}{2\pi}\int_{\mathbb{R}}\hat{g}(\xi)e^{ix\xi}{\rm d}\xi
\]
and Plancherel's formula is 
\[
\sqrt{2\pi}\| g \|_{L_2(\mathbb{R})}= \| \mathcal{F}[g] \|_{L_2(\mathbb{R})}.
\]
We will restrict our attention to \emph{band-limited} data.  To this end, we assume that the data $Y$ are samples taken from a function in the Paley-Wiener space
\[
PW_\pi:=\left\{ f\in L_2(\mathbb{R}): \mathcal{F}[f] = 0 \text{ on } \mathbb{R}\setminus[-\pi,\pi]  \right\}.
\]
That is, $Y:=Y[f(X)]=(f(x_j): j\in\mathbb{Z})$, where $f\in PW_\pi$ and $X$ is a CIS.

Our argument closely follows the one found in \cite{LEDFORD20131}.  However, our splines are not covered under the umbrella of \emph{regular interpolators} found there.  In the Fourier domain, we have that $s\in S_{\alpha}^{k}(X)$ enjoys the representation
\[
\hat{s}(\xi)=(-1)^{k}(\xi^2+\alpha^2)^{-k}\sum_{j\in\mathbb{Z}}a_je^{-ix_j\xi},
\]
in the distributional sense for some $(a_j)\in\ell_2$.

Our next result concerns two operators associated to the interpolation problem, $B_{\alpha,k}, M_{\alpha,k}:L_2([-\pi,\pi])\to L_2([-\pi,\pi])$ which are defined as follows:
\begin{align*}
B_{\alpha,k}[g](\xi)&:= \sum_{n\neq 0}A_n^{*}\left(\left(\left(\cdot+2\pi n\right)^{2}+\alpha^{2}\right)^{-k}A_{n}[g]\right)(\xi)\\
M_{\alpha,k}[g](\xi)&:=(\alpha^2+\xi^2)^{-k}g(\xi).
\end{align*}

\begin{thm}
    Suppose that $\alpha>0$ and $k\in\mathbb{N}$. For any $g \in L_2([-\pi,\pi])$ there exists constants $ C_0,C_1 > 0$, independent of $g$, such that
    \[
    C_0\|g\|^{2}_{L_2([-\pi,\pi])} 
        \leq \left|\left\langle (M_{\alpha,k}+B_{\alpha,k})[g],g\right\rangle\right|
        \leq C_{1}\|g\|^{2}_{L_2([-\pi,\pi])}.
\]

\end{thm}

\begin{proof}
 We begin with the upper bound.
\begin{align*}
\left|\left\langle (M_{\alpha,k}+B_{\alpha,k})[g],g\right\rangle\right| =& \left| \int_{-\pi}^{\pi}(M_{\alpha,k}+B_{\alpha,k})[g](\xi)\overline{g(\xi)} {\rm d}\xi\right|\\
\leq & \sum_{n\in\mathbb{\mathbb{Z}}}\int_{-\pi}^{\pi}|A_{n}^{*}(((\cdot+2\pi n)^{2}+\alpha^{2})^{-k}A_{n}[g])(\xi)\overline{g(\xi)}|{\rm d}\xi\\
\leq & C_X\sum_{n\in\mathbb{\mathbb{Z}}}\int_{-\pi}^{\pi}(\xi + 2\pi n)^{2}+\alpha^{2})^{-k}A_{n}[g](\xi) \overline{g(\xi)}{\rm d}\xi \\
\leq & C_X\sum_{n\in\mathbb{\mathbb{Z}}} \sup_{|t|\leq\pi}(t + 2\pi n)^{2}+\alpha^{2})^{-k}\int_{-\pi}^{\pi}A_{n}[g](\xi)\overline{g(\xi)}{\rm d}\xi\\
\leq & C_X^2\sum_{n\in\mathbb{\mathbb{Z}}} \sup_{|t|\leq\pi}(t + 2\pi n)^{2}+\alpha^{2})^{-k}\int_{-\pi}^{\pi}|g(\xi)|^2{\rm d}\xi\\
\leq & C_X^2 \left(\alpha^{-2k}+\sum_{n\neq 0}((2|n|-1)^2\pi^2+\alpha^2)^{-k} \right)\| g\|_{L_2([-\pi,\pi])}^2\\
\end{align*}

We have used the triangle inequality and then Tonelli's theorem to switch the sum and the integral in the first inequality.  The second and fourth inequality follow from \eqref{Riesz basis inequality}, while the third is monotonicity of the integral.  Finally, the extreme value theorem from Calculus provides the final inequality. 

The lower bound is somewhat easier since the upper bound allows us to use the dominated convergence theorem to interchange the integral and the sum.
\begin{align*}
\left|\left\langle (M_{\alpha,k}+B_{\alpha,k})[g],g\right\rangle\right| =& \left| \int_{-\pi}^{\pi}(M_{\alpha,k}+B_{\alpha,k})[g](\xi)\overline{g(\xi)} {\rm d}\xi\right|\\
=& \left|\sum_{n\in\mathbb{Z}}\int_{-\pi}^{\pi} A_{n}^{*}(((\cdot+2\pi n)^{2}+\alpha^{2})^{-k}A_{n}[g])(\xi)\overline{g(\xi)}{\rm d}\xi  \right|\\
=&\left|\sum_{n\in\mathbb{Z}}\int_{-\pi}^{\pi}((\xi+2\pi n)^{2}+\alpha^{2})^{-k}A_{n}[g](\xi)\overline{A_n[g](\xi)} {\rm d}\xi \right|\\
\geq &\int_{-\pi}^{\pi}(\xi^2+\alpha^{2})^{-k}|g(\xi)|^2 {\rm d}\xi\\
\geq& (\pi^2+\alpha^2)^{-k} \|g \|_{L_2([-\pi,\pi])}^2
    \end{align*}

The second equality may be justified by the dominated convergence theorem, and the third is just the adjoint relation.  The first inequality follows from the positivity of the summands, while the final inequality follows from the monotonicity of the integral.

\end{proof}

Throughout the sequel, we consider an arbitrary but fixed $f\in PW_\pi$ sampled at a fixed but otherwise arbitrary CIS $X=(x_j:j\in\mathbb{Z})$. Since $\hat{f}\in L_2([-\pi,\pi])$, we may define the function $\varphi_f$ via
\begin{equation}\label{phi}
\varphi_{f}(\xi) := (-1)^k(M_{\alpha,k}+B_{\alpha,k})^{-1}\hat{f}(\xi)=\sum_{j\in\mathbb{Z}}c[f]_je^{-ix_j\xi};\qquad |\xi|\leq\pi.
\end{equation}

Consider the function $I[f]$, defined by its Fourier transform:
\begin{equation}\label{hat_I}
    \widehat{I[f]}(\xi):=(-1)^k(\alpha^2+\xi^2)^{-k}\varphi_{f}(\xi).
\end{equation}
Since $\varphi_f\in L_2^{loc}(\mathbb{R})$, we have $\widehat{I[f]}\in (L_1\cap L_2)(\mathbb{R})$, which means that the Fourier inversion theorem applies, and $I[f]\in (L_2\cap C_0)(\mathbb{R})$.  In particular, $I[f](x)$ is well defined and for all $n\in\mathbb{Z}$, we have
\begin{align*}
    2\pi I[f](x_n) =& \int_{\mathbb{R}}\widehat{I[f]}(\xi)e^{i x_n\xi}{\rm d}\xi\\
    =&(-1)^k\int_{-\pi}^{\pi}\sum_{j\in\mathbb{Z}} (\alpha^2+(\xi+2\pi j)^2)^{-k} A_j[\varphi_f](\xi) A_j[e^{ix_n\cdot}](\xi){\rm d}\xi\\
    =&(-1)^k\int_{-\pi}^{\pi}\sum_{j\in\mathbb{Z}} A_j^*\left( (\alpha^2+(\cdot+2\pi j)^2)^{-k} A_j[\varphi_f](\cdot)\right)(\xi) e^{ix_n\xi}{\rm d}\xi\\
    =&\int_{-\pi}^{\pi}(M_{\alpha,k}+B_{\alpha,k})[\varphi_f](\xi)e^{ix_n\xi}{\rm d}\xi \\
    =&\int_{-\pi}^{\pi}\hat{f}(\xi)e^{ix_n\xi}{\rm d}\xi\\
    =&2\pi f(x_n).
\end{align*}

Hence $I[f]$ interpolates $f$ at $X$.

Given the form of $\widehat{I[f]}$, it is clear that $I[f]$ satisfies $I[f]\in C^{2k-2}(\mathbb{R})$ and the ordinary differential equation
\begin{equation}\label{eqn: ODE}
(D^2-\alpha^2 )^k (I[f])=\sum_{j\in\mathbb{Z}}c[f]_j \delta(\cdot-x_j), 
\end{equation}
where $c[f]\in\ell^2$ is chosen by \eqref{phi}.  That is, $I[f]$ is a polyhyperbolic spline. As a consequence of the celebrated Paley-Wiener theorem, we can interpolate any sequence $(a_j)\in\ell^2$ by first finding $f\in PW_\pi$ which satisfies $f(x_j)=a_j$, then using Theorem 1 to get $c[f]\in\ell^2$.  

We summarize these results in the following theorem.
\begin{thm}\label{thm: high degree main}
    Suppose that $b\in\ell^2$ and $X=(x_j)$ is a complete interpolating sequence.  There exists a solution of \eqref{polyhyperbolic}, denoted $I[b]$, that satisfies
    \begin{enumerate}
        \item $I[b]\in (C^{2k-2}_0\cap L_2)(\mathbb{R})$
        \item $I[b](x_n)=b_n$ for all $n\in\mathbb{Z}$.
    \end{enumerate}   
\end{thm}

\begin{remark*}
    If we require that our sampled function $f\in PW_\pi$ satisfies
    \[
    |f(x)|\leq C \cosh(\alpha x),
    \]
    then we may solve the interpolation problem for $t_\alpha\in T_{\alpha}^k(X)$.
\end{remark*}

Having provided a solution to the interpolation problem, we may now follow the outline of \cite{LEDFORD20131,Madych-Lyubarski,Schlumprecht-Sivakumar} to prove the following recovery result concerning polyhyperbolic splines.

\begin{thm}
    Suppose that $f\in PW_\pi$ and $X=(x_j)$ is a CIS.  Let $I_k[f]$ denote the polyhyperbolic interpolant whose Fourier transform is defined in \eqref{hat_I}.  Then we have the following
    \begin{align*}
        a.& \lim_{k\to\infty}\| I_k[f]-f \|_{L_2(\mathbb{R})}=0, \text{ and}\\
        b. &\lim_{k\to\infty}|I_k[f](x)-f(x)|=0,\text{ uniformly on }\mathbb{R}. 
    \end{align*}
\end{thm}


\section{Algorithm}

We include a Mathematica\texttrademark{} source code for $s_\alpha\in S_{\alpha}^2(X)$ for the interested reader.

{
\lstset{
	tabsize=4,
	frame=single,
	language=mathematica,
	basicstyle=\scriptsize\ttfamily,
	keywordstyle=\bfseries,
	backgroundcolor=\color{gris245},
	commentstyle=\color{olive},
	showstringspaces=false,
	emph={[1]x, y, acoe, bcoe, ccoe, dcoe, allcoe, s1, s2, ds, dds, mat1, mat2, yend},emphstyle={[1]\itshape},
	emph={[2]i}, emphstyle={[2]\color{blue}},
	emph={[3]n},emphstyle={[3]\color{purple}},
        emph={[4]p}, emphstyle={[4]\color{teal}}
	}

\begin{lstlisting}
(*Enter in x, y, and p values*)
x = {(*Input list of values*)};
y = {(*Input list of values or functions of x*)};
t = (*Input type of end condition*);
yend = {(*Input end conditions for type 1 or 2*)};
p = (*Input tension parameter*);
n = Length[x];

(*Constructs coefficient lists*)
acoe = Table[Symbol["a" <> ToString[i]], {i, n - 1}];
bcoe = Table[Symbol["b" <> ToString[i]], {i, n - 1}];
ccoe = Table[Symbol["c" <> ToString[i]], {i, n - 1}];
dcoe = Table[Symbol["d" <> ToString[i]], {i, n - 1}];
allcoe = {};
For[i = 1, i <= (n - 1), i++, 
  AppendTo[allcoe, Symbol["a" <> ToString[i]]];
  AppendTo[allcoe, Symbol["b" <> ToString[i]]];
  AppendTo[allcoe, Symbol["c" <> ToString[i]]];
  AppendTo[allcoe, Symbol["d" <> ToString[i]]];
];

(*Creates spline equations and puts them in a list*)
s1 = Array[0 &, n - 1];
s2 = Array[0 &, n - 1];
For[i = 1, i < n, i++, 
  s1[[i]] = ((acoe[[i]] + bcoe[[i]]*x[[i]])*Exp[-p*x[[i]]])
            + ((ccoe[[i]] + dcoe[[i]]*x[[i]])*Exp[p*x[[i]]]);
];
For[i = 1, i < n, i++, 
  s2[[i]] = ((acoe[[i]] + bcoe[[i]]*x[[i + 1]])*Exp[-p*x[[i + 1]]]) 
            + ((ccoe[[i]] + dcoe[[i]]*x[[i + 1]])
            *Exp[p*x[[i + 1]]]);
];

(*Creates the first derivative of spline equations and puts
them in a list*)
ds = Array[0 &, n - 2];
For[i = 1, i <= n - 2, i++, 
  ds[[i]] = (bcoe[[i]]*Exp[-p*x[[i + 1]]] 
            + dcoe[[i]]*Exp[p*x[[i + 1]]] - Exp[-p*x[[i + 1]]]*p
            *(acoe[[i]] + bcoe[[i]]*x[[i + 1]]) 
            + Exp[p*x[[i + 1]]]*p
            *(ccoe[[i]] + dcoe[[i]]* x[[i + 1]])) 
            - (bcoe[[i + 1]]*Exp[-p*x[[i + 1]]] 
            + dcoe[[i + 1]]*Exp[p*x[[i + 1]]] - Exp[-p*x[[i + 1]]]
            *p*(acoe[[i + 1]] + bcoe[[i + 1]]*x[[i + 1]]) 
            + Exp[p*x[[i + 1]]]*p
            *(ccoe[[i + 1]] + dcoe[[i + 1]]* x[[i + 1]]));
];

(*Creates the second derivative of spline equations and puts 
them in a list*)
dds = Array[0 &, n - 2];
For[i = 1, i <= n - 2, i++, 
  dds[[i]] = (-2*bcoe[[i]]*Exp[-p*x[[i + 1]]]*p 
             + 2*dcoe[[i]]*Exp[p*x[[i + 1]]]*p + Exp[-p*x[[i + 1]]]
             *(acoe[[i]] + bcoe[[i]]*x[[i + 1]])*p^2 
             + Exp[p*x[[i + 1]]]*(ccoe[[i]] + dcoe[[i]]*x[[i + 1]])
             *p^2) - (-2*bcoe[[i + 1]]*Exp[-p*x[[i + 1]]]*p 
             + 2*dcoe[[i + 1]]*Exp[p*x[[i + 1]]]*p 
             + Exp[-p*x[[i + 1]]]
             *(acoe[[i + 1]] + bcoe[[i + 1]]*x[[i + 1]])
             *p^2 + Exp[p*x[[i + 1]]] 
             * (ccoe[[i + 1]] + dcoe[[i + 1]]*x[[i + 1]])*p^2);
];

(*Sets up invertible matrix by putting the lists s1, s2, ds, 
and dds in an array*)
mat1 = {};
For[i = 1, i < n, i++, 
  AppendTo[mat1, Coefficient[s1[[i]], allcoe]];
];
For[i = 1, i < n, i++, 
  AppendTo[mat1, Coefficient[s2[[i]], allcoe]];
];
For[i = 1, i <= n - 2, i++, 
  AppendTo[mat1, Coefficient[ds[[i]], allcoe]];
];
For[i = 1, i <= n - 2, i++, 
  AppendTo[mat1, Coefficient[dds[[i]], allcoe]];
];
(*Adds end conditions to matrix*)
If[t == 1, 
  AppendTo[mat1, Coefficient[(bcoe[[1]]*Exp[-p*x[[1]]] 
    + dcoe[[1]]*Exp[p*x[[1]]] 
    - Exp[-p*x[[1]]]*p*(acoe[[1]] + bcoe[[1]]*x[[1]]) 
    + Exp[p*x[[1]]]*p*(ccoe[[1]] + dcoe[[1]]* x[[1]])), 
    allcoe];
  ];
  AppendTo[mat1, Coefficient[(bcoe[[n - 1]]*Exp[-p*x[[n]]] 
    + dcoe[[n - 1]]*Exp[p*x[[n]]] 
    - Exp[-p*x[[n]]]*p*(acoe[[n - 1]] + bcoe[[n - 1]]*x[[n]]) 
    + Exp[p*x[[n]]]*p*(ccoe[[n - 1]] + dcoe[[n - 1]]* x[[n]])), 
    allcoe];
  ];,
  If[t == 2 || t == 3,
    AppendTo[mat1, Coefficient[(-2*bcoe[[1]]*Exp[-p*x[[1]]]*p 
      + 2*dcoe[[1]]*Exp[p*x[[1]]]*p 
      + Exp[-p*x[[1]]]*(acoe[[1]] + bcoe[[1]]*x[[1]])*p^2 
      + Exp[p*x[[1]]]*(ccoe[[1]] + dcoe[[1]]*x[[1]])*p^2),
      allcoe];
    ];
    AppendTo[mat1, Coefficient[(-2*bcoe[[n - 1]]*Exp[-p*x[[n]]]*p 
      + 2*dcoe[[n - 1]]*Exp[p*x[[n]]]*p 
      + Exp[-p*x[[n]]]*(acoe[[n - 1]] + bcoe[[n - 1]]*x[[n]])*p^2 
      + Exp[p*x[[n]]]*(ccoe[[n - 1]] + dcoe[[n - 1]]*x[[n]])*p^2), 
      allcoe];
    ];
  ];
];


(*Creates solution matrix*)
mat2 = {};
For[i = 1, i < n, i++, 
  AppendTo[mat2, y[[i]]];
];
For[i = 2, i <= n, i++, 
  AppendTo[mat2, y[[i]]];
];
For[i = 2*(n - 1), i < (n - 1)*4, i++, 
  AppendTo[mat2, 0];
];

(*Redefines allcoe as a list for all the values of 
the coefficients*)
allcoe = LinearSolve[mat1, mat2];

(*Final spline function*)
f[u_] = Sum[allcoe[[4*k - 3 ;; 4*k]].{Exp[-p*u], 
          u*Exp[-p*u], Exp[p*u], u*Exp[p*u]}
          *Piecewise[{{1, x[[k]] <= u < x[[k + 1]]}}, 0],
          {k, 1, n - 1}];

Plot[f[u], {u, x[[1]], x[[n]]}];

\end{lstlisting}

}


\bibliographystyle{plain} 
\bibliography{main}

\begin{thebibliography}{10}

\bibitem{aggarwal}
Charu~C. Aggarwal.
\newblock {\em Neural networks and deep learning}.
\newblock Springer, Cham, 2018.
\newblock A textbook.

\bibitem{AHLBERG}
J.~H. Ahlberg and E.~N. Nilson.
\newblock Convergence properties of the spline fit.
\newblock {\em J. Soc. Indust. Appl. Math.}, 11:95--104, 1963.

\bibitem{Birkhoff_deBoor}
Garrett Birkhoff and Carl de~Boor.
\newblock Error bounds for spline interpolation.
\newblock {\em J. Math. Mech.}, pages 827--835, 1964.

\bibitem{unser_deep}
Pakshal Bohra, Joaquim Campos, Harshit Gupta, Shayan Aziznejad, and Michael Unser.
\newblock Learning activation functions in deep (spline) neural networks.
\newblock {\em IEEE Open Journal of Signal Processing}, 1:295--309, 2020.

\bibitem{Dougherty}
Randall~L. Dougherty, Alan~S. Edelman, and James~M. Hyman.
\newblock Nonnegativity-, monotonicity-, or convexity-preserving cubic and quintic {H}ermite interpolation.
\newblock {\em Math. Comp.}, 52(186):471--494, 1989.

\bibitem{Hall}
C.~A. Hall.
\newblock On error bounds for spline interpolation.
\newblock {\em J. Approximation Theory}, 1:209--218, 1968.

\bibitem{Ksasov}
Boris~I. Ksasov and Pairote Sattayatham.
\newblock G{B}-splines of arbitrary order.
\newblock {\em J. Comput. Appl. Math.}, 104(1):63--88, 1999.

\bibitem{ledford}
J.~Ledford.
\newblock Polyhyperbolic cardinal splines.
\newblock In {\em Approximation theory {XV}: {S}an {A}ntonio 2016}, volume 201 of {\em Springer Proc. Math. Stat.}, pages 189--196. Springer, Cham, 2017.

\bibitem{LEDFORD20131}
Jeff Ledford.
\newblock Recovery of paley–wiener functions using scattered translates of regular interpolators.
\newblock {\em Journal of Approximation Theory}, 173:1--13, 2013.

\bibitem{Madych-Lyubarski}
Yu. Lyubarski\u{\i} and W.~R. Madych.
\newblock The recovery of irregularly sampled band limited functions via tempered splines.
\newblock {\em J. Funct. Anal.}, 125(1):201--222, 1994.

\bibitem{Maciejewski}
M.~Maciejewski.
\newblock Hyperbolic splines with given derivatives at the knots.
\newblock {\em Applicationes Mathematicae}, 18(2):319--336, 1984.

\bibitem{McCartin}
B.~McCartin.
\newblock Theory of exponential splines.
\newblock {\em J. Approx. Theory}, 66(1):1--23, 1991.

\bibitem{Pruess}
S.~Pruess.
\newblock Properties of splines in tension.
\newblock {\em J. Approximation Theory}, 17(1):86--96, 1976.

\bibitem{Schlumprecht-Sivakumar}
Th. Schlumprecht and N.~Sivakumar.
\newblock On the sampling and recovery of bandlimited functions via scattered translates of the {G}aussian.
\newblock {\em J. Approx. Theory}, 159(1):128--153, 2009.

\bibitem{Schoenberg_Part_A}
I.~J. Schoenberg.
\newblock Contributions to the problem of approximation of equidistant data by analytic functions. {P}art {A}. {O}n the problem of smoothing or graduation. {A} first class of analytic approximation formulae.
\newblock {\em Quart. Appl. Math.}, 4:45--99, 1946.

\bibitem{SCHOENBERG}
I.~J. Schoenberg.
\newblock {\em Cardinal spline interpolation.}
\newblock Society for Industrial and Applied Mathematics, Philadelphia, Pa.,,, 1973.

\bibitem{MR0053177}
I.~J. Schoenberg and Anne Whitney.
\newblock On {P}\'{o}lya frequence functions. {III}. {T}he positivity of translation determinants with an application to the interpolation problem by spline curves.
\newblock {\em Trans. Amer. Math. Soc.}, 74:246--259, 1953.

\bibitem{Schumaker}
Larry~L. Schumaker.
\newblock On hyperbolic splines.
\newblock {\em J. Approx. Theory}, 38(2):144--166, 1983.

\bibitem{Schumaker_book}
Larry~L. Schumaker.
\newblock {\em Spline functions: basic theory}.
\newblock Cambridge Mathematical Library. Cambridge University Press, Cambridge, third edition, 2007.

\bibitem{Spath}
H.~Sp\"{a}th.
\newblock Exponential spline interpolation.
\newblock {\em Computing (Arch. Elektron. Rechnen)}, 4:225--233, 1969.

\bibitem{Taussky}
O.~Taussky.
\newblock A recurring theorem on determinants.
\newblock {\em The American Mathematical Monthly}, 56(10):672--676, 1949.

\bibitem{unser_representer}
Michael Unser.
\newblock A representer theorem for deep neural networks.
\newblock {\em J. Mach. Learn. Res.}, 20:Paper No. 110, 30, 2019.

\end{thebibliography}

\end{document}